\documentclass[12pt]{amsart}
\usepackage{amssymb,hyperref}

\newtheorem{theorem}{Theorem}[section]
\newtheorem{definition}[theorem]{Definition}

\newtheorem{proposition}[theorem]{Proposition}

\newtheorem{problem}[theorem]{Problem}
\newtheorem{lemma}[theorem]{Lemma}

\begin{document}

\title{On orthogonal matrices maximizing the 1-norm}

\author{Teodor Banica}
\address{T.B.: Institut of Mathematics at Toulouse, UMR CNRS 5219, Paul Sabatier University, 31062 Toulouse cedex 9, France. {\tt banica@math.univ-toulouse.fr}}

\author{Beno\^\i{}t Collins}
\address{B.C.: Department of Mathematics, Lyon 1 University, and University of Ottawa, 585 King Edward, Ottawa, ON K1N 6N5, Canada. {\tt bcollins@uottawa.ca}}

\author{Jean-Marc Schlenker}
\address{J.-M.S.: Institut of Mathematics at Toulouse, UMR CNRS 5219, Paul Sabatier University, 31062 Toulouse cedex 9, France. {\tt schlenker@math.univ-toulouse.fr}}

\begin{abstract}
For $U\in O(N)$ we have $||U||_1\leq N\sqrt{N}$, with equality if and only if $U=H/\sqrt{N}$, with $H$ Hadamard matrix. Motivated by this remark, we discuss in this paper the algebraic and analytic aspects of the computation of the maximum of the 1-norm on $O(N)$. The main problem is to compute the $k$-th moment of the 1-norm, with $k\to\infty$, and we present a number of general comments in this direction.
\end{abstract}

\subjclass[2000]{47A30 (05B20)}
\keywords{Orthogonal group, Hadamard matrix}

\maketitle

\section*{Introduction}

The Hadamard conjecture states that for any $N\in 4\mathbb N$, there is an Hadamard matrix of order $N$. This conjecture, known for a long time, and reputed to be of remarkable difficulty, has been verified so far up to $N=664$. See \cite{kta}, \cite{lgo}.

The asymptotic Hadamard conjecture states that for any $N\in 4\mathbb N$ big enough, there is an Hadamard matrix of order $N$. There are of course several problems here, depending on the exact value or theoretical nature of the lower bound. 

The interest in this latter problem comes from the fact that it has several analytic formulations, in terms of random matrices or group integrals. First, since the Hadamard matrices are the $\pm 1$ matrices which maximize the determinant, a natural approach would be via the statistical properties of the determinant of random $\pm 1$ matrices. There are several tools available here, see \cite{nra}, \cite{tvu}.

A second approach comes from an opposite point of view: start with the orthogonal  group $O(N)$, and try to find inside a matrix having $\pm 1/\sqrt{N}$ entries. 

In order to locate these latter matrices, a natural idea would be to use the $p$-norm, $p\neq 2$. Indeed, it follows from the H\"older inequality that the $\pm 1/\sqrt{N}$ matrices maximize the $p$-norm on $O(N)$ at $p<2$, and minimize it at $p>2$.

In this paper we present some results at $p=1$. We work out a number of remarkable algebraic properties of the maximizers of the 1-norm, and we perform as well some analytic computations. Most of our results concern the following key quantity:
$$K_N=\sup_{U\in O(N)}||U||_1$$

With this notation, the various estimates that we obtain, and their relation to the Hadamard conjecture (HC), are as follows:
\begin{center}
\begin{tabular}[t]{|l|l|l|l|l|}
\hline Type&Estimate&Comment\\
\hline
\hline Cauchy-Schwarz&$K_N=N\sqrt{N}$&Equivalent to HC at $N$\\
\hline Cauchy-Schwarz, improved&$K_N\geq N\sqrt{N}-1/(N\sqrt{N})$&Equivalent to HC at $N$\\
\hline Elementary&$K_N\geq (N-4.5)\sqrt{N}$&Implied by HC\\
\hline Spherical integral&$K_N\geq 0.797N\sqrt{N}$&True\\
\hline
\end{tabular}
\end{center}
\medskip

In addition, we prove that we have $K_3=5$: an elementary  result about the usual $3\times 3$ matrices, whose proof is non-trivial, and which seems to be unknown.

The estimate in the last row comes by computing the average over $O(N)$ of the 1-norm. In general, the problem is to compute the $k$-th moment of the 1-norm:
$$I_k=\int_{O(N)}||U||_1^k\,dU$$

Indeed, since $K_N=\lim_{k\to\infty}I_k^{1/k}$, a good knowledge of $I_k$ would allow in principle to derive fine estimates on $K_N$, in the spirit of those in the first rows of the table. The point, however, is that the computation of $I_k$ is a highly non-trivial task.

Summarizing, we obtain in this paper some evidence for the existence of an ``analytic approach'' to the Hadamard conjecture. We do not know if this approach can lead to concrete results, but we intend to further pursue it in a forthcoming paper \cite{bcs}, where we will investigate the various terms appearing in the expansion of $I_k$:
$$I_k=\sum_{i_1\ldots i_k}\sum_{j_1\ldots j_k}\int_{O(N)}|U_{i_1j_1}\ldots U_{i_kj_k}|\,dU$$

Finally, let us mention that some more piece of evidence for an analytic approach to the ``Hadamard matrix count'' comes from the ``magic square count'' performed by Diaconis and Gamburd in \cite{dga}. In fact, the ultimate dream would be that an analytic formula not only gives the existence of Hadamard matrices, but also counts them.

The paper is organized as follows: in 1-2 we discuss the basic problematics, in 3-4 we use various methods from the calculus of variations, and in 5-6 we use a number of integration techniques. The final sections, 7-8, contain a few concluding remarks.

\subsection*{Acknowledgements}

The work of T.B. and B.C. was supported by the ANR grants ``Galoisint'' and ``Granma''. T.B. and J.-M.S. would like to thank the University of Ottawa, where part of this work was done, for its warm hospitality and support.

\section{General considerations}

An Hadamard matrix is a matrix $H\in M_N(\pm 1)$, whose rows are pairwise orthogonal. It is known that if an Hadamard matrix exists, then $N=2$ or $4|N$.

It is conjectured that the Hadamard matrices of order $N$ exist, for any $4|N$. In what follows we will take this conjecture for granted, in the sense that we will avoid the problems solved by it, and we will basically restrict attention to the case $4\!\!\not|N$. 

Needless to say, our interest in the problems investigated below comes precisely from the Hadamard conjecture: that we don't attempt, however, to solve here.

If $H$ is Hadamard, it follows from definitions that $U=H/\sqrt{N}$ is orthogonal. Observe that by changing a row of $H$, we may assume $U\in SO(N)$.

The $1$-norm of a square matrix $U\in M_N(\mathbb R)$ is given by:
$$||U||_1=\sum_{ij=1}^N|U_{ij}|$$

Our starting point is the following observation.

\begin{proposition}
For $U\in O(N)$ we have $||U||_1\leq N\sqrt{N}$, with equality if and only if $U=H/\sqrt{N}$, for a certain Hadamard matrix $H$.
\end{proposition}

\begin{proof}
We use the Cauchy-Schwarz inequality:
$$\sum_{ij=1}^N1\cdot |U_{ij}|\leq\left(\sum_{ij=1}^N1^2\right)^{1/2}\left(\sum_{ij=1}^NU_{ij}^2\right)^{1/2}$$

The left term term being $||U||_1$ and the right term being $N\sqrt{N}$, this gives the estimate. Moreover, the equality holds when the numbers $|U_{ij}|$ are proportional, and since the sum of squares of these numbers is $N^2$, we conclude that we have equality if and only if $|U_{ij}|=1/\sqrt{N}$, i.e. if and only if $H=\sqrt{N}U$ is Hadamard. 
\end{proof}

This suggests the following problem.

\begin{problem}
For $N>2$ not multiple of $4$, what are the matrices $U\in O(N)$ which maximize the $1$-norm on $O(N)$? 
\end{problem}

It is not clear for instance what happens with these matrices when $N$ varies. Another question is whether the extremum value is computable or not. And so on.

The above problem is invariant under the following equivalence relation.

\begin{definition}
Two matrices $U,V\in O(N)$ are called equivalent if one can pass from one to the other by permuting the rows and columns, or by multiplying them by $-1$.
\end{definition}

Observe in particular that by multiplying a single row with a $-1$, we can assume that our matrix maximizing the 1-norm is in $SO(N)$.

\section{A sharp estimate}

At $N=3$ we have the following remarkable result.

\begin{theorem}
For $U\in O(3)$ we have $||U||_1\leq 5$, and this estimate is sharp.
\end{theorem}

\begin{proof}
According to the general remarks in the end of the previous section, we can assume that we have $U\in SO(3)$. We use the Euler-Rodrigues formula:
$$U=\begin{pmatrix}
x^2+y^2-z^2-t^2&2(yz-xt)&2(xz+yt)\\
2(xt+yz)&x^2+z^2-y^2-t^2&2(zt-xy)\\
2(yt-xz)&2(xy+zt)&x^2+t^2-y^2-z^2
\end{pmatrix}$$

Here $(x,y,z,t)\in S^3$ come from the double cover $SU(2)\to SO(3)$. Now in order to obtain the estimate, we linearize. We must prove that for any $x,y,z,t\in\mathbb R$ we have:
\begin{eqnarray*}
&&|x^2+y^2-z^2-t^2|+|x^2+z^2-y^2-t^2|+|x^2+t^2-y^2-z^2|\\
&&+2\left(|yz-xt|+|xz+yt|+|xt+yz|+|zt-xy|+|yt-xz|+|xy+zt|\right)\\
&&\leq 5(x^2+y^2+z^2+t^2)
\end{eqnarray*}

The problem being symmetric in $x,y,z,t$, and invariant under sign changes, we may assume that we have $x\geq y\geq z\geq t\geq 0$. Now if we look at the 9 absolute values in the above formula, in 7 of them the sign is known, and in the remaining 2 ones the sign is undetermined. More precisely, the inequality to be proved is:
\begin{eqnarray*}
&&(x^2+y^2-z^2-t^2)+(x^2+z^2-y^2-t^2)+|x^2+t^2-y^2-z^2|\\
&&+2\left(|yz-xt|+(xz+yt)+(xt+yz)+(xy-zt)+(xz-yt)+(xy+zt)\right)\\
&&\leq 5(x^2+y^2+z^2+t^2)
\end{eqnarray*}

After simplification and rearrangement of the terms, this inequality reads:
$$|x^2+t^2-y^2-z^2|+2|xt-yz|\leq 3x^2+5y^2+5z^2+7t^2-4xy-4xz-2xt-2yz$$

In principle we have now 4 cases to discuss, depending on the possible signs appearing at left. It is, however, easier to proceed simply by searching for the optimal case. 

First, by writing $y=\alpha+\varepsilon,z=\alpha-\varepsilon$ and by making $\varepsilon$ vary over the real line, we see that the optimal case is when $\varepsilon=0$, hence when $y=z$. 

The case $y=z=0$ or $y=z=\infty$ being clear (and not sharp) we can assume that we have $y=z=1$. Thus we must prove that for $x\geq 1\geq t\geq 0$ we have:
$$|x^2+t^2-2|+2|xt-1|\leq 3x^2+8+7t^2-8x-2xt$$

In the case $xt\geq 1$ we have $x^2+t^2\geq 2$, and the inequality becomes $2xt+4x\leq x^2+3t^2+6$, true. In the case $xt\leq 1,x^2+t^2\leq 2$ we get $x^2+1+2t^2\geq 2x$, true again. In the remaining case $xt\leq 1,x^2+t^2\geq 2$ we get $x^2+4+3t^2\geq 4x$, true again.

This finishes the proof of the estimate. Now regarding the maximum, according to the above discussion this is attained at $(xyzt)=(1110)$ or at $(xyzt)=(2110)$, plus permutations. The corresponding matrix is, modulo permutations:
$$A=\frac{1}{3}
\begin{pmatrix}
1&2&2\\
2&1&-2\\
-2&2&-1
\end{pmatrix}$$

For this matrix we have indeed $||A||_1=5$, and we are done.
\end{proof}

\section{Local maximizers}

In this section we find an abstract characterization of the orthogonal matrices locally maximizing the 1-norm. We begin with some lemmas.

\begin{lemma}
If $U\in O(N)$ locally maximizes the $1$-norm, then $U_{ij}\neq 0$ for any $i,j$.
\end{lemma}

\begin{proof}
Assume that $U$ has a 0 entry. By permuting the rows we can assume that this 0 entry is in the first row, having under it a nonzero entry in the second row.

We denote by $U_1,\ldots,U_N$ the rows of $U$. By permuting the columns we can assume that we have a block decomposition of the following type:
$$\begin{pmatrix}U_1\\ U_2\end{pmatrix}
=\begin{pmatrix}
0&0&Y&A&B\\
0&X&0&C&D
\end{pmatrix}$$

Here $X,Y,A,B,C,D$ are certain vectors with nonzero entries, with $A,B,C,D$ chosen such that each entry of $A$ has the same sign as the corresponding entry of $C$, and each entry of $B$ has sign opposite to the sign of the corresponding entry of $D$.

Our above assumption states that $X$ is not the null vector.

For $t>0$ small consider the matrix $U(t)$ obtained by rotating by $t$ the first two rows of $U$. In row notation, this matrix is given by: 
$$U(t)=\begin{pmatrix}
\cos t&\sin t\\
-\sin t&\cos t\\
&&1\\
&&&\ddots\\
&&&&1\end{pmatrix}
\begin{pmatrix}
U_1\\ U_2\\ U_3\\ \ldots\\ U_N
\end{pmatrix}
=\begin{pmatrix}
\cos t\cdot U_1+\sin t\cdot U_2\\ -\sin t\cdot U_1+\cos t\cdot U_2\\ U_3\\ \ldots\\ U_N
\end{pmatrix}$$

We make the convention that the lower-case letters denote the 1-norms of the corresponding upper-case vectors. According to the above sign conventions, we have:
\begin{eqnarray*}
||U(t)||_1
&=&||\cos t\cdot U_1+\sin t\cdot U_2||_1+||-\sin t\cdot U_1+\cos t\cdot U_2||_1+\sum_{i=3}^Nu_i\\
&=&(\cos t+\sin t)(x+y+b+c)+(\cos t-\sin t)(a+d)+\sum_{i=3}^Nu_i\\
&=&||U||_1+(\cos t+\sin t-1)(x+y+b+c)+(\cos t-\sin t-1)(a+d)
\end{eqnarray*}

By using $\sin t=t+O(t^2)$ and $\cos t=1+O(t^2)$ we get:
\begin{eqnarray*}
||U(t)||_1
&=&||U||_1+t(x+y+b+c)-t(a+d)+O(t^2)\\
&=&||U||_1+t(x+y+b+c-a-d)+O(t^2)
\end{eqnarray*}

In order to conclude, we have to prove that $U$ cannot be a local maximizer of the $1$-norm. This will basically follow by comparing the norm of $U$ to the norm of $U(t)$, with $t>0$ small or $t<0$ big. However, since in the above computation it was technically convenient to assume $t>0$, we actually have three cases:

Case 1: $b+c>a+d$. Here for $t>0$ small enough the above formula shows that we have $||U(t)||_1>||U||_1$, and we are done.

Case 2: $b+c=a+d$. Here we use the fact that $X$ is not null, which gives $x>0$. Once again for $t>0$ small enough we have $||U(t)||_1>||U||_1$, and we are done.

Case 3: $b+c<a+d$. In this case we can interchange the first two rows of $U$ and restart the whole procedure: we fall in Case 1, and we are done again.
\end{proof}

\begin{lemma}
Let $S\in M_N(\pm 1)$ and $U\in O(N)$.
\begin{enumerate}
\item $U$ is a critical point of $F(U)=\Sigma S_{ij}U_{ij}$ iff $SU^t$ is symmetric.

\item $U$ is a local maximum of $F$ if and only if $SU^t>0$.
\end{enumerate}
\end{lemma}

\begin{proof}
We use the basic theory of critical points: Lagrange multipliers for the first part, and the Hessian for the second part.

(1) The orthogonal group consists by definition of the zeroes of the following family of polynomials:
$$A_{ij}=\sum_kU_{ik}U_{jk}-\delta_{ij}$$

We know that $U$ is a critical point of $F$ iff $dF\in span(dA_{ij})$. Since $A_{ij}=A_{ji}$, this is the same as asking for the existence of a symmetric matrix $M$ such that:
\begin{eqnarray*}
dF
&=&\sum_{ij}M_{ij}dA_{ij}\\
&=&\sum_{ijk}M_{ij}(U_{ik}dU_{jk}+U_{jk}dU_{ik})\\
&=&\sum_{jk}(MU)_{jk}dU_{jk}+\sum_{ik}(MU)_{ik}dU_{ik}\\
&=&2\sum_{lk}(MU)_{lk}dU_{lk}
\end{eqnarray*}

On the other hand, by differentiating the formula of $F$ we get:
$$dF=\sum_{lk}S_{lk}dU_{lk}$$

We conclude that $U$ is a critical point of $F$ iff there exists a symmetric matrix $M$ such that $S=2MU$. Now by using the asumption $U\in O(N)$, this condition simply tells us that $M=SU^t/2$ must be symmetric, and we are done.

(2) According to the general theory, the Hessian of $F$ applied to a vector $X=UY$ with $Y\in O(N)$, belonging to the tangent space at $U\in O(N)$, is given by:
\begin{eqnarray*}
Hess(F)(X)
&=&\frac{1}{2}Tr(X^t\cdot SU^t\cdot X)\\
&=&\frac{1}{2}Tr(Y^tU^t\cdot SU^t\cdot UY)\\
&=&\frac{1}{2}Tr(Y^t\cdot U^tS\cdot Y)
\end{eqnarray*}

Thus the Hessian is positive definite when $U^tS$ is positive definite, which is the same as saying that $U(U^tS)U^t=SU^t$ is positive definite. 
\end{proof}

\begin{theorem}
A matrix $U\in O(N)$ locally maximizes the $1$-norm if and only if $SU^t>0$, where $S_{ij}={\rm sgn}(U_{ij})$.
\end{theorem}

\begin{proof}
This follows from Lemma 3.1 and Lemma 3.2.
\end{proof}

As a first illustration, in the case $U=H/\sqrt{N}$ with $H$ Hadamard we have $S=H$, and the matrix $SU^t=HH^t/\sqrt{N}=I_n/\sqrt{N}$ is indeed positive.

In the case of the matrix $A$ from section 2, the verification goes as follows:
$$SA^t=\frac{1}{3}
\begin{pmatrix}
1&1&1\\
1&1&-1\\
-1&1&-1
\end{pmatrix}
\begin{pmatrix}
1&2&-2\\
2&1&2\\
2&-2&-1
\end{pmatrix}
=\frac{1}{3}
\begin{pmatrix}
5&1&-1\\
1&5&1\\
-1&1&5
\end{pmatrix}$$

The matrix on the right is indeed positive, its eigenvalues being $1,2,2$.

One approach to Problem 1.2 would be by first studying the local maximizers of the 1-norm. We have here the following results.

\begin{proposition}
The local maximizers of the $1$-norm have the following properties:
\begin{enumerate}
\item They are stable by tensor product.

\item At $N=2$ they maximize the $1$-norm.

\item In general, they don't necessarily maximize the $1$-norm.
\end{enumerate}
\end{proposition}

\begin{proof}
(1) Assume indeed that $U,V$ maximize the 1-norm, and let $S,X$ be the corresponding sign matrices. The tensor product is given by $W_{ia,jb}=U_{ij}V_{ab}$, with sign matrix $Y_{ia,jb}=S_{ij}X_{ab}$, and the result follows from the following computation:
\begin{eqnarray*}
(YW^t)_{ia,jb}
&=&\sum_{kc}Y_{ia,kc}W_{jb,kc}\\
&=&\sum_{kc}S_{ik}X_{ac}U_{jk}V_{bc}\\
&=&\sum_kS_{ik}U_{jk}\sum_cX_{ac}V_{bc}\\
&=&(SU^t)_{ij}(XV^t)_{ab}
\end{eqnarray*}

(2) By using the equivalence relation, we can assume that the matrix is a rotation of angle $t\in [0,\pi/2]$. With $a=\cos t$, $b=\sin t$ we have:
$$SU^t=\begin{pmatrix}1&1\\ -1&1\end{pmatrix}
\begin{pmatrix}a&-b\\ b&a\end{pmatrix}
=\begin{pmatrix}a+b&a-b\\ b-a&a+b\end{pmatrix}$$

Thus we have $a=b$, so $U=S/\sqrt{2}$, and $S$ is Hadamard.

(3) Consider indeed the matrix $U=A\otimes(H/2)$, where $A$ is the $3\times 3$ matrix found in section 2, and $H$ is a $4\times 4$ Hadamard matrix. Then $U$ locally maximizes the 1-norm on $O(12)$, but is not a multiple of an Hadamard matrix. 
\end{proof}

The local maximizers don't seem to be stable under more general operations.

One problem is to compute these local maximizers at $N=3$. Observe that, by using the equivalence relation, we can assume that $U\in O(3)$ has the following sign matrix:
$$S=\begin{pmatrix}1&1&1\\ 1&-1&1\\ 1&1&-1\end{pmatrix}$$

Equivalently, we can assume that $U\in SO(3)$ has the following sign matrix:
$$S=\begin{pmatrix}1&1&1\\ 1&-1&\pm 1\\ 1&1&-1\end{pmatrix}$$

So, in principle we have 2 concrete problems over $SO(3)$ to be solved. However, the computation with the Euler-Rodrigues formula seems to be quite difficult, and we do not have further results in this direction.

\section{Norm estimates}

Our purpose here is to investigate the following quantity:
$$K_N=\sup_{U\in O(N)}||U||_1$$

The motivation comes from the Hadamard conjecture, because we have $K_N\leq N\sqrt{N}$, with equality if and only if there exists an Hadamard matrix of order $N$.

We begin our study with the following observation.

\begin{proposition}
Assume that the Hadamard conjecture holds. Then we have the estimate $K_N\geq (N-4.5)\sqrt{N}$, valid for any $N\in\mathbb N$.
\end{proposition}

\begin{proof}
If $N$ is a multiple of $4$ we can use an Hadamard matrix, and we are done. In general, we can write $N=M+k$ with $4|M$ and $0\leq k\leq 3$, and use an Hadamard matrix of order $N$, completed with an identity matrix of order $k$. This gives:
\begin{eqnarray*}
K_N
&\geq&
M\sqrt{M}+k\\
&\geq&(N-3)\sqrt{N-3}+3\\
&\geq&(N-4.5)\sqrt{N}+3
\end{eqnarray*}

Here the last inequality, proved by taking squares, is valid for any $N\geq 5$. 
\end{proof}

Observe that our method can be slightly improved, by using the $2\times2$ and $3\times 3$ maximizers of the 1-norm at $k=2,3$. This leads to the following estimate:
$$K_N\geq (N-3)\sqrt{N-3}+5$$

Of course, this won't improve the $4.5$ constant in Proposition 4.1.

In the reminder of this section we discuss the following related question, which is of course of great theoretical importance.

\begin{problem}
Which estimates on $K_N$ imply the Hadamard conjecture?
\end{problem}

More precisely, given $N\in\mathbb N$, we know that we have $K_N\leq N\sqrt{N}$, with equality if and only if the Hadamard conjecture holds at $N$. The problem is find a good numerical bound $B_N<N\sqrt{N}$ such that $K_N\geq B_N$ implies the Hadamard conjecture at $N$.

We present here a straighforward approach to the problem. We use the Cauchy-Schwarz inequality and the basic properties of the 2-norm, that we denote $||\ \,||$.

\begin{lemma}
For any norm one vector $U\in\mathbb R^N$ we have the formula
$$||U||_1=\sqrt{N}\left(1-\frac{||U-H||^2}{2}\right)$$
where $H\in\mathbb R^N$ is the vector given by $H_i={\rm sgn}(U_i)/\sqrt{N}$.  
\end{lemma}

\begin{proof}
We have:
\begin{eqnarray*}
||U-H||^2
&=&\sum_i(U_i-{\rm sgn}(U_i)/\sqrt{N})^2\\
&=&\sum_iU_i^2-2|U_i|/\sqrt{N}+1/N\\
&=&||U||^2-2||U||_1/\sqrt{N}+1\\
&=&2-2||U||_1/\sqrt{N}
\end{eqnarray*}

This gives the result.
\end{proof}

\begin{lemma}
Let $N$ be even, and $U\in O(N)$ be such that $H=S/\sqrt{N}$ is not Hadamard, where $S_{ij}={\rm sgn}(U_{ij})$. Then $||U||_1\leq N\sqrt{N}-1/(N\sqrt{N})$.
\end{lemma}

\begin{proof}
Since $H$ is not Hadamard, this matrix has two distinct rows $H_1,H_2$ which are not orthogonal. Since $N$ is even, we must have $|<H_1,H_2>|\geq 2/N$. We get:
\begin{eqnarray*}
||U_1-H_1||+||U_2-H_2||
&\geq&|<U_1-H_1,H_2>|+|<U_2-H_2,U_1>|\\
&\geq&|<U_1-H_1,H_2>+<U_2-H_2,U_1>|\\
&=&|<U_2,U_1>-<H_1,H_2>|\\
&=&|<H_1,H_2>|\\
&\geq&2/N
\end{eqnarray*}

Now by applying Lemma 4.3 to $U_1,U_2$, we get:
\begin{eqnarray*}
||U_1||_1+||U_2||_1
&=&\sqrt{N}\left(2-\frac{||U_1-H_1||^2+||U_2-H_2||^2}{2}\right)\\
&\leq&\sqrt{N}\left(2-\left(\frac{||U_1-H_1||+||U_2-H_2||}{2}\right)^2\right)\\
&\leq&\sqrt{N}\left(2-\frac{1}{N^2}\right)\\
&=&2\sqrt{N}-\frac{1}{N\sqrt{N}}
\end{eqnarray*}

By adding to this inequality the 1-norms of the remaining $N-2$ rows, all bounded from above by $\sqrt{N}$, we obtain the result.  
\end{proof}

\begin{theorem}
If $N$ is even and $K_N\geq N\sqrt{N}-1/(N\sqrt{N})$, the Hadamard conjecture holds at $N$.
\end{theorem}

\begin{proof}
If the Hadamard conjecture doesn't hold at $N$, then the assumption of Lemma 4.4 is satisfied for any $U\in O(N)$, and this gives the result.
\end{proof}

In principle the above estimate can be slightly improved, by using the fact that the local maximizers of the $1$-norm satisfy the condition $SU^t>0$, coming from Theorem 3.3. However, there is probably very small room for improvements, for instance because at $N=3$ we have $N\sqrt{N}-1/(N\sqrt{N})=5.003..$, which is very close to $K_3=5$.

\section{Spherical integrals}

As explained in the introduction, the systematic study of $K_N$ can be done by computing certain integrals on $O(N)$, corresponding to the moments of the 1-norm. 

In this section we compute the simplest such integral, namely the average of the 1-norm. This will lead to the estimate $K_N\geq 0.797N\sqrt{N}$, mentioned in the introduction.

We denote by $S^{N-1}\subset\mathbb R^N$ the usual sphere, with coordinates $x_1,\ldots,x_N$, and taken with the uniform measure of mass 1. 

We use the notation $m!!=(m-1)(m-3)(m-5)\ldots$, with the product ending at 2 if $m$ is odd, and ending at 1 if $m$ is even.

\begin{lemma}
For any $k_1,\ldots,k_p\in\mathbb N$ we have
$$\int_{S^{N-1}}\left|x_1^{k_1}\ldots x_p^{k_p}\right|\,dx=\left(\frac{2}{\pi}\right)^{\Sigma(k_1,\ldots,k_p)}\frac{(N-1)!!k_1!!\ldots k_p!!}{(N+\Sigma k_i-1)!!}$$
with $\Sigma=[odds/2]$ if $N$ is odd and $\Sigma=[(odds+1)/2]$ if $N$ is even, where ``odds'' denotes the number of odd numbers in the sequence $k_1,\ldots,k_p$.
\end{lemma}

\begin{proof}
We use the following notation: $\delta(a,b)=0$ if both $a,b$ are even, and $\delta(a,b)=1$ if not. Observe that we have $\delta(a,b)=[(odds(a,b)+1)/2]$.

As a first observation, the result holds indeed at $N=2$, due to the following well-known formula:
$$\frac{2}{\pi}\int_0^{\pi/2}\cos^pt\sin^qt\,dt=\left(\frac{2}{\pi}\right)^{\delta(p,q)}\frac{p!!q!!}{(p+q+1)!!}$$

Let us discuss now that general case. According to the general theory, the integral in the statement can be written in spherical coordinates, as follows:
$$I=\frac{2^N}{V}\int_0^{\pi/2}\ldots\int_0^{\pi/2}x_1^{k_1}\ldots x_N^{k_N}J\,dt_1\ldots dt_{N-1}$$

Here $V$ is the volume of the sphere, $J$ is the Jacobian, and the $2^N$ factor comes from the restriction to the $1/2^N$ part of the sphere where all the coordinates are positive.

The normalization constant in front of the integral is:
$$\frac{2^N}{V}
=\frac{2^N}{N\pi^{N/2}}\cdot\Gamma\left(\frac{N}{2}+1\right)
=\left(\frac{2}{\pi}\right)^{[N/2]}(N-1)!!$$

As for the unnormalized integral, this is given by:
\begin{eqnarray*}
I'=\int_0^{\pi/2}\ldots\int_0^{\pi/2}
&&(\cos t_1)^{k_1}\\
&&(\sin t_1\cos t_2)^{k_2}\\
&&\ldots\\
&&(\sin t_1\sin t_2\ldots\sin t_{N-2}\cos t_{N-1})^{k_{N-1}}\\
&&(\sin t_1\sin t_2\ldots\sin t_{N-2}\sin t_{N-1})^{k_N}\\
&&\sin^{N-2}t_1\sin^{N-3}t_2\ldots\sin^2t_{N-3}\sin t_{N-2}\\
&&dt_1\ldots dt_{N-1}
\end{eqnarray*}

By rearranging the terms, we get:
\begin{eqnarray*}
I'
&=&\int_0^{\pi/2}\cos^{k_1}t_1\sin^{k_2+\ldots+k_N+N-2}t_1\,dt_1\\
&&\int_0^{\pi/2}\cos^{k_2}t_2\sin^{k_3+\ldots+k_N+N-3}t_2\,dt_2\\
&&\ldots\\
&&\int_0^{\pi/2}\cos^{k_{N-2}}t_{N-2}\sin^{k_{N-1}+k_N+1}t_{N-2}\,dt_{N-2}\\
&&\int_0^{\pi/2}\cos^{k_{N-1}}t_{N-1}\sin^{k_N}t_{N-1}\,dt_{N-1}
\end{eqnarray*}

Now by using the formula at $N=2$, we get:
\begin{eqnarray*}
I'
&=&\frac{\pi}{2}\cdot\frac{k_1!!(k_2+\ldots+k_N+N-2)!!}{(k_1+\ldots+k_N+N-1)!!}\left(\frac{2}{\pi}\right)^{\delta(k_1,k_2+\ldots+k_N+N-2)}\\
&&\frac{\pi}{2}\cdot\frac{k_2!!(k_3+\ldots+k_N+N-3)!!}{(k_2+\ldots+k_N+N-2)!!}\left(\frac{2}{\pi}\right)^{\delta(k_2,k_3+\ldots+k_N+N-3)}\\
&&\ldots\\
&&\frac{\pi}{2}\cdot\frac{k_{N-2}!!(k_{N-1}+k_N+1)!!}{(k_{N-2}+k_{N-1}+k_N+2)!!}\left(\frac{2}{\pi}\right)^{\delta(k_{N-2},k_{N-1}+k_N+1)}\\
&&\frac{\pi}{2}\cdot\frac{k_{N-1}!!k_N!!}{(k_{N-1}+k_N+1)!!}\left(\frac{2}{\pi}\right)^{\delta(k_{N-1},k_N)}
\end{eqnarray*}

In this expression most of the factorials cancel, and the $\delta$ exponents on the right sum up to the following number:
$$\Delta(k_1,\ldots,k_N)=\sum_{i=1}^{N-1}\delta(k_i,k_{i+1}+\ldots+k_N+N-i-1)$$

In other words, with this notation, the above formula reads:
\begin{eqnarray*}
I'
&=&\left(\frac{\pi}{2}\right)^{N-1}\frac{k_1!!k_2!!\ldots k_N!!}{(k_1+\ldots+k_N+N-1)!!}\left(\frac{2}{\pi}\right)^{\Delta(k_1,\ldots,k_N)}\\
&=&\left(\frac{2}{\pi}\right)^{\Delta(k_1,\ldots,k_N)-N+1}\frac{k_1!!k_2!!\ldots k_N!!}{(k_1+\ldots+k_N+N-1)!!}\\
&=&\left(\frac{2}{\pi}\right)^{\Sigma(k_1,\ldots,k_N)-[N/2]}\frac{k_1!!k_2!!\ldots k_N!!}{(k_1+\ldots+k_N+N-1)!!}
\end{eqnarray*}

Here the formula relating $\Delta$ to $\Sigma$ follows from a number of simple observations, the first of which is the following one: due to obvious parity reasons, the sequence of $\delta$ numbers appearing in the definition of $\Delta$ cannot contain two consecutive zeroes.

Together with $I=(2^N/V)I'$, this gives the formula in the statement.
\end{proof}

As a first observation, the exponent $\Sigma$ appearing in the statement of Lemma 5.1 can be written as well in the following compact form:
$$\Sigma(k_1,\ldots,k_p)=\left[\frac{N+odds+1}{2}\right]-\left[\frac{N+1}{2}\right]$$

However, for concrete applications, the writing in Lemma 5.1 is more convenient.

\begin{theorem}
With $N\to\infty$ we have $K_N\geq\sqrt{2/\pi}\cdot N\sqrt{N}$.
\end{theorem}

\begin{proof}
We use the well-known fact that the row slices of $O(N)$ are all isomorphic to the sphere $S^{N-1}$, with the restriction of the Haar measure of $O(N)$ corresponding in this way to the uniform measure on $S^{N-1}$. Together with a standard symmetry argument, this shows that the average of the 1-norm on $O(N)$ is given by:
\begin{eqnarray*}
\int_{O(N)}||U||_1\,dU
&=&\sum_{ij}\int_{O(N)}|U_{ij}|\,dU\\
&=&N^2\int_{O(N)}|U_{11}|\,dU\\
&=&N^2\int_{S^{N-1}}|x_1|\,dx
\end{eqnarray*}

We denote by $I$ the integral on the right. According to Lemma 5.1, we have:
\begin{eqnarray*}
I
&=&\left(\frac{2}{\pi}\right)^{\Sigma(1)}\frac{(N-1)!!}{N!!}\\
&=&\begin{cases}
\displaystyle{\frac{2}{\pi}\cdot\frac{2.4.6\ldots(N-2)}{3.5.7\ldots(N-1)}}& (N\ {\rm even})\\
\displaystyle{1\cdot\frac{3.5.7\ldots(N-2)}{2.4.6\ldots(N-1)}}& (N\ {\rm odd})
\end{cases}\\
&=&
\begin{cases}
\displaystyle{\frac{4^M}{\pi M}}\begin{pmatrix}2M\\ M\end{pmatrix}^{-1}& (N=2M)\\
4^{-M}\begin{pmatrix}2M\\ M\end{pmatrix}& (N=2M+1)
\end{cases}
\end{eqnarray*}

Now by using the Stirling formula, we get:
\begin{eqnarray*}
I
&\simeq&\begin{cases}
\displaystyle{\frac{4^M}{\pi M}\cdot\frac{\sqrt{\pi M}}{4^M}}& (N=2M)\\
\displaystyle{4^{-M}\cdot\frac{4^M}{\sqrt{\pi M}}}& (N=2M+1)
\end{cases}\\
&=&\begin{cases}
\displaystyle{\frac{1}{\sqrt{\pi M}}}& (N=2M)\\
\displaystyle{\frac{1}{\sqrt{\pi M}}}& (N=2M+1)
\end{cases}\\
&\simeq&\sqrt{\frac{2}{\pi N}}
\end{eqnarray*}

Thus the average of the 1-norm is asymptotically equal to $\sqrt{2/\pi}\cdot N\sqrt{N}$. Now since the maximum of the 1-norm is greater than its average, we get the result.
\end{proof}

Observe that the constant appearing in the above statement is $\sqrt{2/\pi}=0.797..$. Thus the above result justifes the last row of the table in the introduction.

\section{Higher moments}

In order to find better estimates on $K_N$, the problem is to compute the higher moments of the 1-norm, given by:
$$I_k=\int_{O(N)}||U||_1^k\,dU$$

In general, this is a quite delicate problem. As an illustration for the potential difficulties appearing here, let us work out the case of the second moment. 

\begin{theorem}
With $N\to\infty$ we have $I_2\geq (1+4/\pi)N^2$.
\end{theorem}

\begin{proof}
We use Lemma 5.1 for the spherical part, and the inequality $|U_{11}U_{22}|\geq U_{11}^2U_{22}^2$ together with the Weingarten formula in \cite{csn} for the non-spherical part. We get:
\begin{eqnarray*}
I_2
&=&N^2\int_{O(N)}\!U_{11}^2dU+2N^2(N-1)\int_{O(N)}\!|U_{11}U_{12}|dU+N^2(N-1)^2\int_{O(N)}\!|U_{11}U_{22}|dU\\
&\geq&N^2\int_{S^{N-1}}x_1^2\,dx+2N^2(N-1)\int_{S^{N-1}}|x_1x_2|\,dx+N^2(N-1)^2\int_{O(N)}U_{11}^2U_{22}^2\,dU\\
&=&N^2\cdot\frac{1}{N}+2N^2(N-1)\cdot\frac{2}{\pi}\cdot\frac{1}{N}+N^2(N-1)^2\cdot\frac{N+1}{(N-1)N(N+2)}
\end{eqnarray*}

This gives the estimate in the statement.
\end{proof}

The above proof makes it clear that the complexity of the computation basically grows exponentially with $k$. The point, however, is that some simplifications should appear in the limit $k\to\infty$, which corresponds to the key computation.

It is beyond the purposes of this paper to further develop this point of view. Observe however that the $k\to\infty$ simplifications expected to appear would not be part of the usual Weingarten philosophy \cite{wei}, which roughly states that some remarkable simplifications appear in the limit $N\to\infty$. However, in order to approach the asymptotic Hadamard conjecture via the present analytic techniques, both $k\to\infty$ and $N\to\infty$ methods and simplifications are probably needed. There seems to be a lot of work to be done here, and we intend to clarify a bit the situation in our next paper \cite{bcs}.

\section{Generalizations}

We would like to point out here that some useful, alternative approaches to the problem might come from the use of the $p$-norm, with $p\neq 2$ arbitrary. 

The $p$-norm of a square matrix $U\in M_N(\mathbb R)$ is given by:
$$||U||_p=\left(\sum_{ij=1}^N|U_{ij}|^p\right)^{1/p}$$

We have the following result, generalizing Proposition 1.1.

\begin{proposition}
Let $U\in O(N)$, and $p\in [1,\infty]$.
\begin{enumerate}
\item If $p<2$ then $||U||_p\leq N^{2/p-1/2}$, with equality iff $H=\sqrt{N}U$ is Hadamard.

\item If $p>2$ then $||U||_p\geq N^{2/p-1/2}$, with equality iff $H=\sqrt{N}U$ is Hadamard.
\end{enumerate}
\end{proposition}

\begin{proof}
We use the H\"older inequality, and the equality $||U||_2=\sqrt{N}$.

(1) In the case $p<2$, the estimate follows from:
$$\sum_{ij=1}^N1\cdot|U_{ij}|^p\leq\left(\sum_{ij=1}^N1^{2/(2-p)}\right)^{1-p/2}\left(\sum_{ij=1}^N(|U_{ij}|^p)^{2/p}\right)^{p/2}$$

(2) In the case $p>2$, the estimate follows from:
$$\sum_{ij=1}^N1\cdot U_{ij}^2\leq\left(\sum_{ij=1}^N1^{p/(p-2)}\right)^{1-2/p}\left(\sum_{ij=1}^N(U_{ij}^2)^{p/2}\right)^{2/p}$$
 
In both cases the equality holds when all the numbers $|U_{ij}|$ are proportional, and we conclude that we have equality if and only if $|U_{ij}|=1/\sqrt{N}$, as stated. 
\end{proof}

Observe that at $p=1$, we recover indeed Proposition 1.1. Observe also that at $p=4$ we get that $||U||_4\geq 1$, with equality if and only if $\sqrt{N}U$ is Hadamard.

The various problems discussed in this paper make sense for any $p\neq 2$.

\begin{problem}
For $N>2$ not multiple of $4$, what are the matrices $U\in O(N)$ which maximize the $p$-norm $(p<2)$, or minimize the $p$-norm $(p>2)$? 
\end{problem}

In addition, it is not clear what happens with the local or global maximizers or minimizers of the $p$-norm, when $N$ is fixed and $p\neq 2$ varies.

\section{Concluding remarks}

We have seen in this paper that the Hadamard conjecture suggests the study of the $p$-norm on $O(N)$. The key quantity to be computed is the supremum of the norm:
$$K_N=\sup_{U\in O(N)}\left(\sum_{ij=1}^N|U_{ij}|^p\right)^{1/p}$$

Moreover, we have seen that there is a natural analytic approach to the approximate computation of $K_N$, based on the study of the following integral:
$$I_k=\int_{O(N)}\left(\sum_{ij=1}^N|U_{ij}|^p\right)^{k/p}\,dU$$

The exponent $p=1$, used in this paper, seems to be the most adapted for the general study of the norm, and notably for the computation of the local maximizers. For integral computations the exponent $p=4$ seems to be adapted as well. 

Some further advances should come via a combination of the various $p=1$ and $p=4$ techniques. We intend to come back to these questions in a forthcoming paper \cite{bcs}.

\end{document}